\documentclass{article}
\usepackage[all]{xy}         
\usepackage{amsmath}
\usepackage{amsfonts}
\usepackage{amssymb}
\usepackage{amsthm}

\newtheorem{theorem}{Theorem}[section]
\newtheorem{lemma}[theorem]{Lemma}
\newtheorem{cor}[theorem]{Corollary}
\newtheorem{conj}[theorem]{Conjecture}
\newtheorem{prop}[theorem]{Proposition}
\newtheorem{thm}[theorem]{Theorem}
\newtheorem{main}[theorem]{Main Result}

\theoremstyle{definition}
\newtheorem{defi}[theorem]{Definition}

\theoremstyle{remark}
\newtheorem{rmk}[theorem]{Remark}

\numberwithin{equation}{section}

\DeclareMathOperator{\coker}{coker}
\newcommand{\isom}{\cong}

\newcommand{\Z}{{\mathbb{Z}}}

\newcommand{\Q}{{\mathbb{Q}}}

\newcommand{\R}{{\mathbb{R}}}
\newcommand{\C}{{\mathbb{C}}}

\newcommand{\Mid}{|} 

\newcommand{\miD}{|}

\newcommand{\ra}{{\rightarrow}}

 1
\DeclareFontEncoding{OT2}{}{} 
  \newcommand{\textcyr}[1]{%
    {\fontencoding{OT2}\fontfamily{wncyr}\fontseries{m}\fontshape{n}%
     \selectfont #1}}
\newcommand{\Sha}{{\mbox{\textcyr{Sh}}}} 

\newcommand{\Emin}{{E_{\scriptscriptstyle{\rm min}}}}
\newcommand{\Epmin}{{E'_{\scriptscriptstyle{\rm min}}}}

\newcommand{\ED}{{E_{\scriptscriptstyle{-D}}}}

\begin{document}

\title{Periods of quadratic twists of elliptic curves}
\author{Vivek Pal\footnote{Florida State University, the author was funded by the FSU Office of National Fellowships} \\ \\ with an appendix by Amod Agashe\footnote{Amod Agashe was supported by the National Security Agency Grant number Hg8230-10-1-0208}}
\date{}
\maketitle

\begin{abstract} In this paper we prove a relation between the period of an elliptic curve and the period of its real and imaginary quadratic twists. This relation is often misstated in the literature.
\end{abstract}

\section{Introduction}
\par One of the central conjectures in Number Theory is the Birch and Swinnerton-Dyer Conjecture, which predicts how one can obtain arithmetic information from the $L$-function. A simpler question, is to ask:
 \begin{quote}
	(*)  If an elliptic curve satisfies the Birch and Swinnerton-Dyer Conjecture then will its (quadratic) twist also satisfy the Birch and Swinnerton-Dyer  Conjecture? 
\end{quote}
\par Part two of the Birch and Swinnerton-Dyer Conjecture involves many elliptic curve invariants, namely the order of the Tate-Shafarevich group, the period and the order of the torsion subgroup among other important invariants. In this paper we relate the period of an elliptic curve with the period of its quadratic twists. A relation between the orders of the torsion subgroups has already been proven in~\cite{MR1430007}. If a similar result can be drawn for all of the other elliptic curve invariants involved in Part II of the Birch and Swinnerton-Dyer Conjecture then one can prove idea (*). Furthermore, a relation between the arithmetic component group and the regulator, of an elliptic curve and its twist would provide a conjecture for the relation between the order of the Shafarevich-Tate group for an elliptic curve and its twist.

\par One advantage of idea~(*) comes from the fact that quadratic twists of elliptic curves have very different ranks from the original curve. Currently Part two of the Birch and Swinnerton-Dyer Conjecture is only known to be true for families of elliptic curves, usually of low rank; using twists one could possibly extend these results to many different ranks. 

\par In general if $F$ is an elliptic curve, then we denote the invariant differential on~$F$ by $\omega(F)$. We will call a global minimal Weierstrass equation of an elliptic curve simply a minimal equation or minimal model, and denote a minimal model of an elliptic curve $F$ by $F_{\text{min}}$. 

\par Let $E$ be an elliptic curve. We use the Birch and Swinnerton-Dyer definition of the period. Recall that this period, denoted by $\Omega(E)$, is defined as:
$$\Omega(E) := \int _{E_{\text{min}} (\mathbb{R})} {|\omega(E_{\text{min}})|}.$$ 

\par Also, recall that the imaginary period, defined up to a sign, is $$\Omega^-(E) := \int _{\gamma^-} {\omega(E_{\text{min}})},$$ where $\gamma^-$ is a generator of $H_1(E_{\text{min}},\mathbb{Z})^-$, which is the subgroup of elements in $H_1(E_{\text{min}},\mathbb{Z})$ which are negated by complex conjugation.

\par Furthermore, recall that the quadratic twist of an elliptic curve $E$ by a non-zero integer $d$, denoted by $E^d$, is defined as an elliptic curve which is isomorphic to $E$ over $\mathbb{Q}(\sqrt{d})$ but not over $\mathbb{Q}$. Hence we can assume that $d$ is square-free. We also know that $E^d$ is unique up to isomorphism.

\par The main result of this paper is then 

\begin{main} \label{thm:main} Let $E$ be an elliptic curve and 
let $E^d$ denote its quadratic twist by a square-free integer~$d$. Then the periods of $E$ and $E^d$ are related as follows: \\
If $d >0$, then
$$ \Omega(E^d)=  \frac{\tilde{u}}{\sqrt{d}} \Omega(E), $$
and if $d<0$, then up to a sign,
$$ \Omega(E^d)=  \frac{\tilde{u}}{\sqrt{d}} c_\infty(E^d) \Omega^-(E) , $$
where $c_\infty(E^d)$ is the number of connected components of $E^d(\mathbb{R})$ and $\tilde{u}$ is a rational number such that $2 \tilde{u} \in \Z$; it depends on~$E$ and~$d$, and is defined explicitly in Proposition~\ref{minimal} (the elliptic curve~$E$ in Proposition~\ref{minimal} should be taken as a minimal model of the $E$ in this theorem). 
\end{main}

The main result above is proved as Theorem~\ref{period} below.

\begin{rmk} $\tilde{u}$ is not always 1 and $2 \tilde{u}$ can be divisible by an odd prime number. In Section~4, we give an example where $\tilde{u}$ is $5$ and an example where $\tilde{u}$ is~$7$. Also, in the last paragraph of the appendix (Section~\ref{sec:app}), there is an example of an {\em optimal} elliptic curve for which $\tilde{u}$ has positive 3-adic valuation.
\end{rmk}

\par A result similar to the second case of the theorem above was derived in~\cite[Lemma~2.1]{vissq} for elliptic curves in short Weierstrass form using an assumption on which primes one can twist by. The result here is proved without restrictions.

\par The main result of this paper allows for a weaker hypothesis for several results in \cite{vissq}; the details are discussed in the appendix.

\par We would like to remark that the formulas in the Main Result have been stated incorrectly in the literature. For example, Amod Agashe informed the author that they are quoted without $\tilde{u}$ as formula~(12) on p.~463 in the proof of Corollary~3 in~\cite{ono-skinner-annals}; he also mentioned that the proof of Corollary~3 in~\cite{ono-skinner-annals} still works even after the formula is corrected to include $\tilde{u}$.

The author would like to thank Amod Agashe for suggesting this problem and for his help in revising several drafts of this paper. Furthermore the reference to Connell's book \cite{connell} was mentioned to the author by Amod Agashe, who in turn heard about it from Randy Heaton. 

\section{Quadratic Twists and Minimal Models}

\par First we recall some useful facts. An elliptic curve over $\mathbb{Q}$ can be described in the following general Weierstrass form:
$$ y^2 + a_1xy +a_3y = x^3 + a_2x^2 + a_4x + a_6,$$
with $a_1, a_2,a_3,a_4,a_6 \in \mathbb{Q}$.

\par  In this paper, by an elliptic curve, we mean a curve given by a Weierstrass equation. An elliptic curve will be called minimal if its Weierstrass equation is minimal. Let $E$ be an elliptic curve, and let $\Delta(E), j(E), c_4(E)$ and $c_6(E)$ be the usual Weierstrass invariants of the elliptic curve $E$. Then the signature of the elliptic curve $E$ is the triple $c_4(E), c_6(E), \Delta(E)$. If $p$ is a prime, then letting $v_p$ denote the standard $p$-adic valuation, the $p$-adic signature of $E$ is the triple $v_p(c_4(E)), v_p(c_6(E)), v_p(\Delta(E))$.

\begin{rmk} \label{prop} A transformation $E \rightarrow E'$ of elliptic curves over $\mathbb{Q}$ preserving the Weierstrass equation and the point at infinity is given by:\\
 \begin{center} $ x= u^2x' +r$  and  $y= u^3y' +u^2sx' +t ,$ \end{center}
for some $u,r,s,t \in \mathbb{Q}$. We will often abbreviate this transformation as the ordered tuple $[u,r,s,t]$. Such a  transformation has the following useful properties:
\begin{enumerate}
\item $u^4c_4(E') = c_4(E)$
\item $u^6 c_6(E')=c_6(E)$
\item $u^{12}  \Delta(E') = \Delta(E)$
\item $j(E')=j(E)$
\item $\omega(E') = u \omega(E)$
\end{enumerate}
\end{rmk}

\par The above facts can be found in any standard book on elliptic curves, for example see Silverman~\cite{silverman:aec}.

\par Since the period of an elliptic curve depends only on the isomorphism class, for the purpose of proving Main Result~\ref{thm:main} or for computing $\Omega(E)$, we can assume that $E$ is a minimal model, i.e. $E=E_{\text{min}}$. So henceforth, let $E$ be an elliptic curve given by the minimal equation: \begin{equation*} \tag{$E$}  y^2 + a_1xy +a_3y = x^3 + a_2x^2 + a_4x + a_6.  \end{equation*}

\begin{lemma}[Connell] \label{twistformula} Let $d$ be a square-free integer. Then a Weierstrass equation for $E^d$ is: 

\begin{align*} \tag{$E^d$} y^2 &+ a_1xy +a_3y= \\ &= x^3 + \left(a_2 d + a_1^2 \frac{d-1}{4}\right) x^2 + \left(a_4 d^2 + a_1 a_3 \frac{d^2-1}{2} \right)x + \left(a_6 d^3 + a_3^2 \frac{d^3-1}{4}\right). \end{align*}
\end{lemma}
\begin{proof} See \cite[ Proposition 4.3.2]{connell} and the paragraph preceding it.
\end{proof}

\begin{rmk} \label{twist} The signature for elliptic curve $(E^d)$ is: $c_4(E^d) = c_4(E) \cdot d^2$, $c_6(E^d) = c_6(E) \cdot d^3$ and $\Delta(E^d) = \Delta(E) \cdot d^6$. Let $\alpha=\sqrt{1/d}$. Then the transformation from $E$ to $E^d$ is: \begin{equation*} \left\{ \begin{split} &x = \alpha^2 x' \\  &y = \alpha^3 y' + \frac{a_1 \alpha^2(\alpha-1)}{2} x' + \frac{a_3 (\alpha^3-1)}{2}. \end{split} \right. \end{equation*}

\end{rmk}

Next we recall a proposition from Connell which is displayed below for convenience. It describes $v_p(\Delta)$ for a minimal model of the twist for each prime $p$.

\begin{prop}[Connell] \label{connell} Recall that $E$ is a minimal elliptic curve over $\mathbb{Q}$ and $E^d$ is its quadratic twist by a square-free integer $d$. Let $\Delta$ be the discriminant of $E$, let $\Delta'$ be the discriminant of $E^d_{\text{min}}$, and for every valuation $v$ on $\mathbb{Z}$ let $\lambda_v =\min\{3v(c_4(E)),2v(c_6(E)), v(\Delta)\}$. If $p$ is a prime number, then let $v_p$ denote the standard $p$-adic valuation. Then
\begin{enumerate}

\item If $p$ is an odd prime divisor of $d$ then:
	\begin{enumerate}
		\item  If $\lambda_{v_p} <6$ or if $p=3$ and $v_p(c_6(E)) =5$, then $v_p(\Delta') = v_p(\Delta) + 6$.
		\item Otherwise $v_p(\Delta') = v_p(\Delta) - 6$.
	\end{enumerate}
If $p$ is an odd prime not dividing $d$, then $v_p(\Delta') = v_p(\Delta)$.
\item If $p=2$ then: 
\begin{enumerate} 
		\item If $d \equiv 1 \mod 4$, then $v_2(\Delta') =v_2(\Delta)$.
		\item If $d \equiv 3 \mod 4$, then 
	\begin{enumerate}	
			\item  If the 2-adic signature of $E$ is $0,0,c$ $(c\geq 0)$ or $a,3,0$ $(4 \leq a \leq \infty)$, then $v_2(\Delta') =v_2(\Delta)+12$.
			\item  If the 2-adic signature of $E$ is $4,6,c$ $(c\geq 12 \text{  and  } 2^{-6}c_6(E)d\equiv -1 \mod 4)$ or $a,9,12$ $(a \geq 8 \text{  and  } 2^{-9}c_6(E)d \equiv 1 \mod 4)$, then $v_2(\Delta') =v_2(\Delta)-12$.		
			\item Otherwise $v_2(\Delta') =v_2(\Delta)$.
	\end{enumerate}
		\item If $d \equiv 2 \mod 4$, let $w=d/2$ then 
	\begin{enumerate}		
			\item  If the 2-adic signature of $E$ is $0,0,c$ $(c\geq 0)$, then $v_2(\Delta') =v_2(\Delta)+18$.
			\item  If the 2-adic signature of $E$ is $6,9,c$ with $(c \geq 18)$ and $2^{-9} c_6(E) w \equiv -1 \mod 4$, then $v_2(\Delta') =v_2(\Delta)-18$.
			\item  If $v_2(c_4(E))=4,5$ or $v_2(c_6(E))=3,5,7$ or the 2-adic signature of $E$ is $a,6,6$ with $(a \geq 6)$ and $2^{-6} c_6(E) w \equiv -1 \mod 4$, then $v_2(\Delta') =v_2(\Delta)+6$.
			\item Otherwise $v_2(\Delta') =v_2(\Delta) -6$.
	\end{enumerate}
\end{enumerate}
\end{enumerate}
\end{prop}
\begin{proof} This proposition is the corrected form of \cite[5.7.3]{connell}, which is misstated in Connell's book. The proof  given by Connell in \cite[5.7.1]{connell} is however correct. This was pointed out to the author by the referee.
\end{proof}

\begin{prop} \label{minimal} Recall that $E$ is a minimal elliptic curve over $\mathbb{Q}$ and $E^d$ is its quadratic twist by a square-free integer~$d$. Let $\Delta$ be the discriminant of $E$, let $\Delta'$ be the discriminant of $E^d_{\text{min}}$, and for every valuation~$v$ on~$\Z$, let $\lambda_v =\min\{3v(c_4(E)),2v(c_6(E)),v(\Delta)\}$.  If $p$ is a prime number, then let $v_p$ denote the standard $p$-adic valuation. Define $u_p$ for all primes $p$, as follows (the cases correspond exactly to the cases of Proposition~\ref{connell}):

\begin{enumerate}

\item If $p$ is an odd prime divisor of $d$, then:
	\begin{enumerate}
	\item If $\lambda_{v_p} <6$ or if $p=3$ and $v_p(c_6(E)) =5$, then $u_p = 1$.
	\item Otherwise $u_p = p$.
	\end{enumerate}
If $p$ is an odd prime not dividing $d$, then $u_p=1$.
\item If $p=2$ then: 
\begin{enumerate} 
		\item If $d \equiv 1 \mod 4$, then $u_2 = 1$.
		\item If $d \equiv 3 \mod 4$, then 
	\begin{enumerate}	
			\item If the 2-adic signature of $E$ is $0,0,c$ $(c\geq 0)$ or $a,3,0$ $(4 \leq a \leq \infty)$, then $u_2 = 1/2$.
			\item  If the 2-adic signature of $E$ is $4,6,c$ $(c\geq 12 \text{  and  } 2^{-6}c_6(E)d\equiv -1 \mod 4)$ or $a,9,12$ $(a \geq 8 \text{  and  } 2^{-9}c_6(E)d \equiv 1 \mod 4)$, then $u_2 = 2$.
			\item Otherwise $u_2 = 1$.
	\end{enumerate}
		\item If $d \equiv 2 \mod 4$, let $w=d/2$, then  
	\begin{enumerate}		
			\item If the 2-adic signature of $E$ is $0,0,c$ $(c\geq 0)$, then $u_2 = 1/2$.
			\item If the 2-adic signature of $E$ is $6,9,c$ with $(c \geq 18)$ and $2^{-9} c_6(E) w \equiv -1 \mod 4$, then $u_2 = 4$.
			\item If $v_2(c_4(E))=4,5$ or $v_2(c_6(E))=3,5,7$ or the 2-adic signature of $E$ is $a,6,6$ with $(a \geq 6)$ and $2^{-6} c_6(E) w \equiv -1 \mod 4$, then $u_2 = 1$.
			\item Otherwise $u_2 = 2$.
	\end{enumerate}
\end{enumerate}
\end{enumerate}
Let $\tilde{u}=\prod_p{ u_p}$. Then there exist $r,s,t \in \mathbb{Q}$ such that the transformation $[\tilde{u},r,s,t]$ will transform equation $E^d$ to a minimal model. 

\end{prop}

\begin{proof} \par The idea of the proof is to apply Proposition~\ref{connell} to the elliptic curve $E$ and then to find a transformation sending $E^d$ to a minimal model. 
\par We claim that $[\tilde{u},0,0,0]$ transforms $E^d$ to a curve with the correct minimal discriminant. This follows on a case by case basis using Proposition~\ref{connell}, Remark~\ref{twist}, and Remark~\ref{prop}. Take for example the case~$1(b)$: this is the case where, by Proposition~\ref{connell}, $v_p(\Delta(E^d_{\text{min}})) = v_p(\Delta(E)) - 6$. By Remark~\ref{twist}, we know that $v_p(\Delta(E^d))=v_p(d^6\Delta(E))=v_p(\Delta(E))+6$, since in this case, $p$ divides $d$ (and $d$ is square-free). Therefore $v_p(\Delta(E^d_{\text{min}})) = v_p(\Delta(E^d)) -12$. The transformation which will decrease the valuation of the discriminant by 12 is $[p,0,0,0]$  by Remark~\ref{prop}; hence proving the Proposition in this case. Applying a similar process to the other cases will derive the respective $u_p$. Since the $u_p$s are coprime to each other, composing the transformations $[u_p,0,0,0]$ will give the transformation $[\tilde{u},0,0,0]$. Thus the transformation, $[\tilde{u},0,0,0]$, will send $E^d$ to an elliptic curve $E'$ with the correct minimal discriminant, but which may not have integer coefficients.

\par We will now show that we can find $r,s,t \in \mathbb{R}$ so that the transformation $[\tilde{u},r,s,t]$ applied to $E^d$ also gives an integral model for $E^d$, and therefore a minimal model. Since $E^d_{\text{min}} \isom E'$, we know that there is a transformation $[u,r,s,t]$ that sends $E'$ to $E^d_{\text{min}}$ \cite[Cor. 7.8.3]{silverman:aec}. By comparing discriminants we see that $u=\pm 1$; we can assume $u=1$ since we can compose this morphism with $[-1,0,0,0]$ to change the sign of $u$. Composing the morphism $[\tilde{u},0,0,0]$ with $[1,r,s,t]$ gives the desired morphism, $[\tilde{u},r,s,t]$, sending $E^d$ to $E^d_{\text{min}}$.
\end{proof}
\par For the benefit of the reader we remark that often the transformation $[\tilde{u},0,0,0]$ will in fact transform $E^d$ to an equation with integral coefficients, hence a minimal model, but for our purposes only the $u$ coefficient of the transformation will play a role later. 

\begin{cor} \label{cor:tildeu} 
We use the notation of Proposition~\ref{minimal}. 
Suppose $d$ is coprime to~$\Delta$.
Then $\tilde{u}$ is a power of~$2$.
Moreover if $d \equiv 1 \mod 4$, then $\tilde{u}=1$.
\end{cor}

\begin{proof} 
Let $p$ be an odd prime. If $p$ does not divide~$d$, then by 
Proposition~\ref{minimal}, $u_p = 1$.
If $p$ divides~$d$, then $v_p(\Delta)=0$ since $d$ is coprime to $\Delta$, and so
$\lambda_{v_p}<6$, and thus 
by Proposition~\ref{minimal}, $u_p = 1$. In both cases, $u_p =1$ for odd primes, which proves
the first claim of the corollary.
. 
If  $d \equiv 1 \mod 4 $, then by Case 2(a) Proposition~\ref{minimal} $u_2=1$. 
The second claim of the corollary follows, since
$\tilde{u}=\prod_p{u_p}=1$
\end{proof}

\begin{defi} We define $E^d_{\text{min}}$ to be the specific minimal model of elliptic curve $E^d$ obtained via Proposition~\ref{minimal}.
\end{defi}

\section{Periods} \label{sec:per}

We first prove a relation between the invariant differentials of $E$ and $E^d_{\text{min}}$ and then use this relation to prove the desired relation between the periods in our main result.

\begin{lemma} \label{omega} We have:
$$ \omega(E^d) = \frac{\omega(E)}{\sqrt{d}}$$ and
$$ \omega(E^d_{\text{min}}) = \frac{\tilde{u}\cdot \omega(E)}{ \sqrt{d}}.$$

\end{lemma}

\begin{proof} Using the properties listed in Remarks~\ref{prop} and \ref{twist} regarding transformations, the transformation taking $E$ to $E^d$ has $u=\alpha=\sqrt{1/d}$. Then by Remark~\ref{prop},
$ \omega(E^d) = \frac{\omega(E)}{\sqrt{d}}$. By Proposition~\ref{minimal}, the transformation taking $E^d$ to $E^d_{\text{min}}$ has $u=\tilde{u}$. Then $ \omega(E^d_{\text{min}}) =  \tilde{u} \cdot \omega(E^d) = \frac{\tilde{u} \cdot \omega(E)}{ \sqrt{d}}$.
\end{proof}

\par We now prove the main result relating the periods.

\begin{thm} \label{period} Recall that $E$ is a minimal elliptic curve and $E^d$ is its quadratic twist by $d$. Then the periods of $E$ and $E^d$ are related as follows \\
If $d >0$, then
$$ \Omega(E^d)=  \frac{\tilde{u}}{\sqrt{d}} \Omega(E). $$
If $d<0$, then up to a sign,
$$ \Omega(E^d)=  \frac{\tilde{u}}{\sqrt{d}} c_\infty(E^d) \Omega^-(E) , $$
where $c_\infty(E^d)$ is the number of connected components of $E^d(\mathbb{R})$.
\end{thm}

\begin{proof}
We first prove the formula for $d>0$:\\

As remarked in the proof of Lemma~\ref{omega}, the transformation that takes $E$ to $E^d$ takes $\omega(E)$ to $\sqrt{d} \omega(E^d)$. This transformation sends $E(\mathbb{R})$ bijectively to $E^d(\mathbb{R})$ because the transformation and its inverse are defined over $\mathbb{R}$ (since $d>0$). Then:

\begin{equation} \label{firstintegral} \int _{E(\mathbb{R})} {|\omega(E)|} = \sqrt{d} \int _{E^d(\mathbb{R})} { |\omega(E^d)|}. \end{equation}
Using a similar argument we see that:

\begin{equation} \label{secondintegral} \int _{E^d(\mathbb{R})} {|\omega(E^d)|}  = \frac{1}{\tilde{u}} \int _{E^d_{\text{min}}(\mathbb{R})} { |\omega(E^d_{\text{min}})|}. \end{equation}

\par Then from equation (\ref{firstintegral}) and equation (\ref{secondintegral}) we see that:

$$ \Omega(E) = \int _{E(\mathbb{R})} {|\omega(E)|} = \frac{\sqrt{d}}{\tilde{u}} \Omega(E^d).$$

Next we prove the formula for $d<0$: \\
We follow the technique used in the proof of \cite[Lemma~2.1]{vissq}. Let $P=(x,y) \in E(\mathbb{R})$ and let $\sigma$ be the complex conjugation map; then $\sigma(P)=P$. The inverse of the map described in Remark~\ref{twist} is given by: 
\begin{equation*} \left\{ \begin{split} &x'=\frac{1}{\alpha^2}x \\ & y'=\frac{1}{\alpha^3}y - \frac{a_1}{2}\left(\frac{1}{\alpha^2}-\frac{1}{\alpha^3} \right)x -\frac{a_3}{2}\left( 1-\frac{1}{\alpha^3} \right)  \end{split} \right. \end{equation*}
where $\alpha=\sqrt{1/d}$. Let $T$ be this map, $T:E^d\rightarrow E$. \\

\emph{Claim:}  $\sigma(T(P)) = -T(P)$. 
\begin{proof}
$$\sigma(T(P)) =  \sigma \left( \left(\frac{1}{\alpha^2}x,\frac{1}{\alpha^3}y - \frac{a_1}{2}\left(\frac{1}{\alpha^2}-\frac{1}{\alpha^3} \right)x -\frac{a_3}{2}\left( 1-\frac{1}{\alpha^3} \right) \right) \right) =$$ $$=\left(\frac{1}{\alpha^2}x,\frac{-1}{\alpha^3}y - \frac{a_1}{2}\left(\frac{1}{\alpha^2}+\frac{1}{\alpha^3} \right)x -\frac{a_3}{2}\left( 1+\frac{1}{\alpha^3} \right) \right).$$
 Using the definition of the negative of a point on an elliptic curve, given in \cite[III.2.3]{silverman:aec}:$$-T(P)=- \left(\frac{1}{\alpha^2}x,\frac{1}{\alpha^3}y - \frac{a_1}{2}\left(\frac{1}{\alpha^2}-\frac{1}{\alpha^3} \right)x -\frac{a_3}{2}\left( 1-\frac{1}{\alpha^3}  \right) \right)=$$ $$=\left(\frac{1}{\alpha^2}x,-\left( \frac{1}{\alpha^3}y - \frac{a_1}{2}\left(\frac{1}{\alpha^2}-\frac{1}{\alpha^3} \right)x -\frac{a_3}{2}\left( 1-\frac{1}{\alpha^3} \right) \right) -a_1\left(\frac{1}{\alpha^2}x\right)-a_3 \right) $$ $$=\left(\frac{1}{\alpha^2}x,\frac{-1}{\alpha^3}y - \frac{a_1}{2}\left(\frac{1}{\alpha^2}+\frac{1}{\alpha^3} \right)x -\frac{a_3}{2}\left( 1+\frac{1}{\alpha^3} \right) \right).$$
Then we see that $\sigma(T(P)) = -T(P)$. 
\end{proof}
Thus $T$ gives a homeomorphism from $E^d(\R)$ to $E(\C)^-$, where $E(\C)^-$ is the subgroup of points not fixed under complex conjugation. If $G$ is a Lie group, then let $G_0$ denote the connected component of~$G$
containing the identity. Then $T$ also induces a homeomorphism from $E^d(\R)_0$ to $E(\C)^-_0$. 

In particular, $T$ gives an isomorphism from $H_1(E^d(\R)_0,\Z)$ to $H_1(E(\C)^-_0,\Z)$. 
By Lemma 4.4 in \cite{agst:bsd}, the natural map 
from $H_1(E^d(\R)_0,\Z)$ to~$H_1(E^d(\C),\Z)^+$ is an isomorphism, and
by Lemma~\ref{amod:lem2} from the appendix (Section~\ref{sec:app}), the natural map 
from $H_1(E^d(\C)^-_0,\Z)$ to~$H_1(E^d(\C),\Z)^-$
is an isomorphism.
Let $\gamma$ be a generator of~$H_1(E^d(\C), \mathbb{Z})^+$. Then from the statements above, one sees that 
$T(\gamma)$ is in $H_1(E(\C),\mathbb{Z})^-$ and generates it.

Then it follows that  
\begin{equation} \label{thirdintegral} \int _{\gamma} {\omega(E^d)}=\int _{T(\gamma)} {T(\omega(E^d))}
= \frac{1}{\sqrt{d}} \int _{T(\gamma)} {\omega(E)} = 
\frac{1}{\sqrt{d}} \Omega^-(E), \end{equation}
where the last equality is up to a sign.

\par Similar to equation (\ref{secondintegral}) we have, 
\begin{equation} \label{fourthintegral} \int _{E^d(\mathbb{R})} {\omega(E^d)}  = \frac{1}{\tilde{u}} \int _{E^d_{\text{min}}(\mathbb{R})} { \omega(E^d_{\text{min}})}, \end{equation}
since the transformation in this integral involves only real numbers it takes $E^d(\mathbb{R})$ to $E^d_{\text{min}}(\mathbb{R})$.
\par Using equation~(\ref{eqn:app}) from the appendix 
and equation (\ref{fourthintegral}), we see
that up to a sign,
\begin{eqnarray} \label{eqn:first}
\Omega(E^d) = \int _{E^d_{\text{min}}(\mathbb{R})} { \omega(E^d_{\text{min}})} = 
\tilde{u} \int _{E^d(\mathbb{R})} {\omega(E^d)} .
\end{eqnarray}
The proof of Lemma~\ref{amod:lem1} from the appendix shows that
\begin{eqnarray} \label{eqn:second}
\int _{E^d(\mathbb{R})} {\omega(E^d)} =
c_\infty(E^d) \int _{\gamma} {\omega(E^d)}.
\end{eqnarray}
Putting equation~(\ref{eqn:second}) in equation~(\ref{eqn:first}), we see that up to a sign:
$$
\Omega(E^d) =  \tilde{u} \cdot 
c_\infty(E^d) \int _{\gamma} {\omega(E^d)} = \frac{\tilde{u}}{\sqrt{d}} \cdot c_\infty(E^d) \cdot \Omega^-(E),
$$
where the last equality follows from equation (\ref{thirdintegral}).
This finishes the proof for the case $d <0$ and proves the theorem.
\end{proof}

\section{Examples}
\subsection{Real quadratic twist}
Using Sage and GP/Pari we were able to find the following example
in which the $\tilde{u}$ in Theorem~\ref{period} is $5$.

Let $E$ be the following elliptic curve $$E: y^2 = x^3 - x^2 - 6883x + 222137,$$
which is minimal.

By Proposition~\ref{minimal}, twisting $E$ by $d=5$ falls in cases 1(b) and 2(a), and so $\tilde{u}=5$. Then by Theorem~\ref{period}, $\Omega(E^d) / \Omega(E) =\frac{5}{\sqrt{5}}=\sqrt{5}$. We now try to verify this in GP/Pari. 

Using Lemma~\ref{twistformula}, we compute the twist by $d=5$ to be $$E^d: y^2 = x^3 - 5x^2 - 172075x + 27767125.$$ 

Using the command ellminimalmodel in GP/Pari we see that one of the minimal models for $E^d$ is then  $$ y^2 = x^3 + x^2 - 275x + 1667 .$$

For an elliptic curve $E$ we can compute the periods in GP/Pari using the command E.omega[1].

\begin{rmk} \label{realperiod} The period computed this way is similar to the period we use, but instead of using a minimal model it is defined as $\int _\gamma {\omega(E)}$, where $\gamma$ is a generator of~$H_1(E(\C), \mathbb{Z})^+$. Therefore we have to first compute a minimal model for~$E$, use that to compute the period in GP/Pari, and then multiply the result by $c_\infty(E)$, the number of connected components, to get the period we desire. 
\end{rmk}

We can see that both $E$ and $E^d$ have only one connected component, by either plotting them or noticing that they both have negative discriminants, thus $c_\infty(E)=c_\infty(E^d)=1$.

 Then one finds that $$ \Omega(E^d) = \Omega(E^d_{\text{min}}) \approx 2.90253993995\dots$$
$$ \Omega(E) \approx 1.29805532262\dots$$

So $\Omega(E^d)/\Omega(E) \approx \sqrt{5}$, 
as expected.

\subsection{How the complex period of GP/Pari relates to the imaginary period defined above.}

Recall that the imaginary period is defined up to a sign as $$\Omega^-(E) = \int _{\gamma} {\omega(E)},$$ where $\gamma$ is a generator of $H_1(E_{\text{min}},\mathbb{Z})^-$. It will be a pure imaginary number since, if $\sigma$ denotes complex conjugation 
$$\sigma(\Omega^-(E)) = \int _{\sigma(\gamma)} {\sigma(\omega(E))}=\int _{-\gamma} {\omega(E)}=-\Omega^-(E).$$ The second equality holds since $\omega(E)$ is defined over $\mathbb{R}$ (in fact over $\mathbb{Q}$) and because $\sigma(\gamma) = -\gamma$.

\par The complex period computed by GP/Pari (using the command E.omega) is in general not a pure imaginary number. Using the periods given by GP/Pari we can however approximately recover the imaginary period. This is because the two periods computed by GP/Pari (called the real and complex periods) are generators for a lattice, which is also generated by the two periods used in this paper (called the period and the imaginary period). For an elliptic curve $E$, let $\Omega_1$ and $\Omega_1^-$ be the period and imaginary period, respectively, defined in this paper. Let $\Omega_2$ and $\Omega_2^-$ be the real and complex periods, respectively, that are computed in GP/Pari for $E$ using the function E.omega. Since the pairs are generators for the same lattice we have, $\Omega_1^- = k_1\Omega_2^- -k_2 \Omega_2$ for some $k_1,k_2 \in \mathbb{Z}$. We also know that $\Omega_1^-$ is a pure imaginary number and that $\Omega_2$ is a real number, therefore  $k_2/k_1=Re(\Omega_2^-)/\Omega_2$ where $Re(z)$ denotes the real part of the complex number $z$. Then $k_1$ and $k_2$ can be chosen such that $\gcd(k_1, k_2) = 1$. Finding such a $k_1$ and $k_2$ gives a way to compute the imaginary period using GP/Pari; however, we can only compute $Re(\Omega_2^-)/\Omega_2$ approximately and hence we can only make a good guess of what $k_1$ and~$k_2$ are.

\subsection{Imaginary quadratic twist}
Using Sage and GP/Pari we were able to find the following example
in which the $\tilde{u}$ in Theorem~\ref{period} is $7$.

Let $E$ be the following elliptic curve $$E:  y^2 + xy + y = x^3 - 173x + 879,$$
which is minimal.

By Proposition~\ref{minimal},  twisting $E$ by $d=-7$ falls in cases 1(b) and 2(a); hence $\tilde{u}=7$. Then by Theorem~\ref{period}, $\Omega(E^d) / \Omega^-(E) = \frac{7}{\sqrt{-7}}=\sqrt{-7}$, up to a sign. We now try to verify this in GP/Pari. 

Using Lemma~\ref{twistformula} we compute the twist by $d=-7$ to be $$E^d: y^2 + xy + y = x^3 - 2x^2 - 8453x - 301583.$$

Using the command ellminimalmodel in GP/Pari we see that one of the minimal models for $E^d$ is then  $$  y^2 + xy = x^3 + x^2 - 3x - 4  .$$

We can see that both $E$ and $E^d$ have only one connected component, by either plotting them or noticing that they both have negative discriminants, thus $c_\infty(E)=c_\infty(E^d)=1$.

Using Remark~\ref{realperiod} one finds that $$\Omega(E^d) = \Omega(E^d_{\text{min}}) \approx  1.73968697697\dots$$
Following the procedure to compute the imaginary period from Section~4.2 we find that $k_2/k_1 \approx -.50000000000\dots$. Assuming that this is actually $-1/2$, we get
$$ \Omega^-(E) \approx (.65753987145\dots) \sqrt{-1}$$  
and  
$\Omega(E^d)/\Omega^-(E) \approx \sqrt{-7}$, 
as expected.

\section{Appendix on periods by Amod Agashe}  \label{sec:app}
In Section~\ref{sec:facts}, we
state and prove some facts about periods that are well
known, but whose proofs do not seem to be documented in the literature;
some of these results are used in Section~\ref{sec:per}.
In Section~\ref{sec:lemma}, we 
give a lemma that is used in Section~\ref{sec:per}. 
In Section~\ref{sec:impl}, we
point out the implications of
the results of this article to~\cite{vissq}, and in particular,
we make a conjecture that strengthens a conjecture made in~\cite{vissq}.

\subsection{Some facts about periods} \label{sec:facts}

Let $E$ be an elliptic curve over~$\Q$, and let $\Emin$
denote an elliptic curve given by a global minimal Weierstrass equation for~$E$.
Let $\omega(\Emin)$ denote the invariant differential on~$\Emin$. Then recall that the
period of~$E$ is defined as 
$$\Omega(E) = \int_{\Emin(\R)} | \omega(\Emin) |.$$
Note that if we take a different global minimal Weierstrass equation for~$E$,
call it $\Epmin$, then $\Emin$ and~$\Epmin$ are isomorphic to each other over~$\Q$
by a transformation of the type $[u,r,s,t]$ 
(notation as in Remark~\ref{prop}) with $u = \pm 1$ (since
they have the same discriminant, and the transformation changes the discriminant
by a factor of~$u^{12}$, by Remark~\ref{prop}). Then the invariant differential
of~$\Epmin$ differs from that of~$\Emin$ by a factor of~$u$ (again, see
Remark~\ref{prop}), i.e., by~$\pm 1$, and so the definition of~$\Omega(E)$
given above is independent of the choice 
of a global minimal Weierstrass equation for~$E$.
If two elliptic curves are
isomorphic over~$\Q$, then they have a common minimal model, and hence
they have the same period.

The N\'eron model of~$E$ is the open subscheme of~$\Emin$
consisting of the regular points (see~\S~III.6 of~\cite{lang:nt3}), and
so the period defined above agrees with the period used in the 
more general version of the Birch and
Swinnerton-Dyer conjecture for abelian varieties (as described for example 
in \S~III.5 of loc. cit.), which uses N\'eron differentials.

Now as a Lie group, $\Emin(\R)$ is isomorphic to one or two copies of~$\R/\Z$
(see, e.g., \cite[Cor. V.2.3.1]{silverman:aec2}).
Since the invariant differential has no zeros or poles 
(see Prop.~III.1.5 in~\cite{silverman:aec}), it does not change its sign on 
any copy of~$\R/\Z$, and so on any copy, we have $|\omega(\Emin)| = \pm \omega(\Emin)$. If $\Emin(\R)$ consists of one copy, then
we see that up to a sign, $\Omega(E) =  \int_{\Emin(\R)} \omega(\Emin)$.
Now suppose $\Emin(\R)$ consists of two copies; call them $C_1$ and~$C_2$. 
Without loss of generality, assume that $C_1$ contains the identity, and
choose a point~$P$ on~$C_2$. Then the translation by~$P$ map induces
a map from~$C_1$ to~$C_2$
(by continuity arguments) and similarly, translation by~$-P$ maps $C_2$ to~$C_1$. These
two maps are inverses to each other, and moreover, $\omega(\Emin)$ is invariant
under translation. Thus we see that the integral of $|\omega(\Emin)|$ over~$C_1$
is the same as that over~$C_2$ and up to a sign is the integral of $\omega(\Emin)$ 
over either component. Thus in this case, 
up to a sign, $\Omega(E) =  2 \cdot \int_{C_1} \omega(\Emin)$.
In either case, we see that up to a sign,
\begin{equation} \label{eqn:app}
\Omega(E) =  \int_{\Emin(\R)}  \omega(\Emin). 
\end{equation}

If $\phi: E \ra \Emin$ is an isomorphism (such an isomorphism exists,
of course), then $\phi$ maps $E(\R)$ bijectively to~$\Emin(\R)$ and
one sees (by integration by substitution) that
$$\int_{\Emin(\R)} \omega(\Emin) = \int_{E(\R)} \phi^*\omega(\Emin),$$
where $\omega(E)$ as usual is the invariant differential on~$E$
and $\phi^*$ denotes the pullback by~$\phi$ map on differentials.
Thus up to a sign,
$$\Omega(E) = \int_{E(\R)}  \phi^* \omega(\Emin).$$
This definition was used in~\cite{agst:manin}, for example.

Considering that $C_1$ is homeomorphic to the circle, and
the natural map 
from the first homology group of~$C_1$ to~$H_1(\Emin(\C),\Z)^+$ is an isomorphism
(e.g., see Lemma 4.4 in~\cite{agst:bsd}),
from the discussion two paragraphs above, we 
get the following lemma:

\begin{lemma}  \label{amod:lem1}
Let $\gamma$ be a generator of the cyclic free abelian group~$H_1(\Emin(\C), \Z)^+$
and let $c_\infty(\Emin)$ denote the number of connected components in~$\Emin(\R)$.
Then up to a sign,
$$\Omega(E) = c_\infty(\Emin) \int_\gamma \omega(\Emin). $$
\end{lemma}

Note that since $E$ and~$\Emin$ are isomorphic over~$\Q$, and hence over~$\R$, 
we have $c_\infty(\Emin) = c_\infty(E)$, and so we also have that up to a sign,
$$\Omega(E) = c_\infty(E) \int_\gamma \omega(\Emin). $$
Lemma~\ref{amod:lem1} above is well known, and in fact a more general result for abelian
varieties is given as Lemma 8.8 in~\cite{manin:cyclo}. However, in loc. cit., the author
only gives a sketch of the proof of the quoted lemma, and uses the result of
Lemma~\ref{amod:lem1} above as an input without proof.

\subsection{A lemma} \label{sec:lemma}

In this section, let $E$ be an elliptic curve over~$\R$. 
Recall that $E(\C)^-$ denotes the subgroup of~$E(\C)$ on which complex conjugation
acts as multiplication by~$-1$ and 
$E(\C)^-_0$ is the component of~$E(\C)^-$ containing the identity.
The following lemma is an adaptation of Lemma 4.4 in~\cite{agst:bsd},
and is used in Section~3.
\begin{lemma} \label{amod:lem2}
The natural map from $H_1(E(\C)^-_0,\Z)$ to~$H_1(E(\C),\Z)^-$
is an isomorphism.
\end{lemma}
\begin{proof}
Let $\psi$ denote the natural map from $H_1(E(\C)^-_0,\Z)$ to~$H_1(E(\C),\Z)^-$.
We have the commutative diagram
$$\xymatrix@=1.3pc{
0 \ar[r] & H_1(E(\C)^-_0,\Z) \ar[r] \ar[d]^{\psi} & H_1(E(\C)^-_0,\R) \ar[r] \ar[d] & E(\C)^-_0 \ar[r] \ar@{^(->}[d]&  0 \\
0 \ar[r] & H_1(E(\C),\Z)^- \ar[r]                    & H_1(E(\C),\R)^- \ar[r] & E(\C)^- 
} \ ,$$
where the two vertical arrows on the right are the obvious natural maps,
the upper horizontal sequence the exact sequence obtained by viewing the real torus~$E(\C)^-_0$ 
as the quotient of the tangent space at the identity by the first integral homology,
and the lower horizontal sequence is the exact sequence 
obtained from the exact sequence
$$
 0\ra H_1(E(\C),\Z) \ra H_1(E(\C),\R) \ra E(\C) \ra 0
$$ 
of complex analytic parametrization of~$E$
by taking anti-invariants under complex conjugation.
The middle vertical map is an isomorphism of real vector spaces because
if it were not, then its kernel would be an uncountable
set that maps to~$0$ in $E(\C)^-_0$ (using the rightmost square in the
commutative diagram above), and hence would be contained
in~$H_1(E(\C)^-_0,\Z)$, which is countable.
The snake lemma then yields an exact sequence
$$ 
   0 \to \ker(\psi) \to 0 \to 0 \to \coker(\psi) \to 0,
$$
which implies that~$\psi$ is an isomorphism, as was to be shown.
\end{proof}

\subsection{Some implications} \label{sec:impl}

In this section, we point out the implications of
the results of this article to~\cite{vissq}.

By Corollary~\ref{cor:tildeu}, if 
$d$ is coprime to the conductor~$E$
(or the discriminant of~$E$), then the~$\tilde{u}$ in Theorem~\ref{period}
is a power of~$2$. Note that the $D$ in~\cite{vissq} is~$-d$, with $d < 0$.
Thus if one replaces the hypothesis~(**) in \S 2 of loc. cit., with the hypothesis
that $D$ is coprime to the conductor~$N$ of~$E$, then the conclusions of
Lemma~2.1, Proposition~2.2, and Corollary~2.4 are valid up to a power of~$2$. 
As a consequence (see the discussion after Corollary~2.4 in loc. cit.), we would
like to weaken the hypothesis~(**) in Conjecture~2.5 of loc. cit. to the hypothesis
that $D$ is coprime to~$N$, and thus make the following
conjecture:

\begin{conj}
Let $E$ be an optimal elliptic curve over~$\Q$ of conductor~$N$ and
let $-D$ be a negative fundamental discriminant such that $D$ is coprime to~$N$.
Recall that $\ED$ denotes the twist of~$E$ by~$-D$. Suppose 
$L(\ED,1) \neq 0$, so that $\ED(\Q)$ is finite.
Then 
$\Mid \ED(\Q) \miD^2$ divides
${\Mid \Sha(\ED) \miD \cdot   \prod_{p | N} c_p(\ED)}$,
up to a power of~$2$,
where $\Sha(\ED)$ denotes 
the Shafarevich-Tate group of~$\ED$ and $c_p(\ED)$ denotes the 
order of the arithmetic component group of~$\ED$ at~$p$.
\end{conj}

As mentioned in loc. cit., 
using the mathematical software sage, with its inbuilt Cremona's database 
for all elliptic curves of conductor up to~$130000$,
we verified the conjecture above for all triples $(N, E, D)$ such 
that $N$ and~$D$ are positive integers with $N D^2 < 130000$, 
and $E$ is an optimal elliptic curve
of conductor~$N$. 

Finally, we remark that Proposition~\ref{minimal} explains why 
the concluding statement of Conjecture~2.5 
of~\cite{vissq} does not hold in the example of $(E, D) = (27a1, 3)$ 
in Table~1 of loc. cit.
(this example does not satisfy the hypotheses of the conjecture):
using SAGE, we find that $\Delta(\Emin) = - 3^9$ and $c_6(\Emin) = 2^3 \cdot 3^6$,
and so by part~1(b) of Proposition~\ref{minimal}, $v_3(\tilde{u}) > 0$. 
In particular, this is an example of an {\em optimal} elliptic curve
for which $\tilde{u}$ is not a power of~$2$.
Anyhow, the concluding statement of Corollary~2.4 in loc. cit. does not hold,
and so for this pair $(E,D)$, assuming the second part of the Birch and Swinnerton-Dyer conjecture,
one does not expect that
$\Mid {E_{\scriptscriptstyle{-D}}}(\Q) \miD^2$ divides
${\Mid \Sha({E_{\scriptscriptstyle{-D}}}) \miD \cdot   \prod_{p | N} c_p({E_{\scriptscriptstyle{-D}}})}$,
even up to a power of~$2$ (see the discussion just before Corollary~2.5 in loc. cit.);
rather one expects that  
$\Mid {E_{\scriptscriptstyle{-D}}}(\Q) \miD^2$ divides
$\tilde{u} \cdot {\Mid \Sha({E_{\scriptscriptstyle{-D}}}) \miD \cdot   \prod_{p | N} c_p({E_{\scriptscriptstyle{-D}}})}$,
and so it is not surprising that 
$\Mid {E_{\scriptscriptstyle{-D}}}(\Q) \miD^2$ divides
$3 \cdot {\Mid \Sha({E_{\scriptscriptstyle{-D}}}) \miD \cdot   \prod_{p | N} c_p({E_{\scriptscriptstyle{-D}}})}$, up to a power of~$2$.

\bibliography{biblio}

\end{document}